\documentclass[11pt]{article}

\usepackage[T1]{fontenc}
\usepackage[utf8]{inputenc}
\usepackage{lmodern}
\usepackage[a4paper,margin=1in]{geometry}
\usepackage{amsmath,amssymb,amsthm,mathtools}
\usepackage{hyperref}
\usepackage{enumitem}

\hypersetup{
  colorlinks=true,
  linkcolor=blue,
  citecolor=blue,
  urlcolor=blue
}

\newtheorem{theorem}{Theorem}[section]
\newtheorem{proposition}[theorem]{Proposition}
\newtheorem{lemma}[theorem]{Lemma}
\newtheorem{corollary}[theorem]{Corollary}
\theoremstyle{definition}
\newtheorem{definition}[theorem]{Definition}
\newtheorem{example}[theorem]{Example}
\theoremstyle{remark}

\title{Local Morphology of the Partition Graph}
\author{Fedor B.~Lyudogovskiy}
\date{}

\begin{document}

\maketitle

\begin{abstract}
For a fixed integer \(n\), let \(G_n\) be the graph whose vertices are the partitions of \(n\), with adjacency defined by a single elementary transfer of a cell in the Ferrers diagram. In a previous paper, the clique complex \(K_n=\mathrm{Cl}(G_n)\) was studied from a global homotopy-theoretic point of view. This paper studies instead the local combinatorics of the graph \(G_n\) itself.

For a partition
\[
\lambda=(s_1^{m_1},\dots,s_t^{m_t}),
\qquad s_1>\dots>s_t>0,
\]
we describe the admissible transfers from \(\lambda\) in terms of its block structure. This yields a bipartite graph \(B(\lambda)\) obtained from \(K_{t,t+1}\) by deleting two explicitly determined families of edges, corresponding to singleton support blocks and unit support gaps. We prove that the graph induced on the neighborhood of \(\lambda\) in \(G_n\) is isomorphic to the line graph \(L(B(\lambda))\).

As consequences, we obtain an explicit formula for the degree of \(\lambda\), a classification of all cliques through \(\lambda\), and a formula for the maximal dimension of a simplex of \(K_n\) containing \(\lambda\). These local invariants are shown to depend only on an ordered binary datum associated with the support of \(\lambda\). The results provide a local structural description of the partition graph and a combinatorial language for the study of larger-scale features of \(G_n\).
\end{abstract}



\section{Introduction}

Graphs on integer partitions arise when partitions are viewed as vertices connected by elementary transformations. In this paper we consider the \emph{partition graph} \(G_n\), whose vertices are the partitions of \(n\), and whose edges correspond to elementary transfers of one cell in a Ferrers diagram. More precisely, one removes a removable corner, adds a cell at an addable corner, and then reorders. This is the graph model used in the author's previous paper on the clique complex
\[
K_n:=\mathrm{Cl}(G_n),
\]
where the global homotopy type of \(K_n\) was studied; see especially Section~2 and Theorems~3.4 and~3.6 of \cite{Lyudogovskiy2026Clique}.

Here we shift the emphasis from the global topology of \(K_n\) to the local combinatorics of the graph \(G_n\) itself. The basic question is the following: given a partition \(\lambda\vdash n\), how much of the local structure of \(G_n\) at the vertex \(\lambda\) can be described in terms of simple intrinsic parameters of \(\lambda\)? Here ``local structure'' means, at minimum, the set of admissible elementary transfers from \(\lambda\), the degree of \(\lambda\), the graph induced on the neighborhood of \(\lambda\), and the clique structure through \(\lambda\).

This question is suggested by the local clique analysis already developed in \cite{Lyudogovskiy2026Clique}. In that paper, the global analysis of \(K_n\) depended on a classification of triangles and cliques in \(G_n\) via two canonical local patterns, called star- and top-simplices; see Section~3 of \cite{Lyudogovskiy2026Clique}. The present work isolates this local aspect and treats it as a primary object of study.

At the same time, the elementary transfer operation itself is not new. In the language of dominance and majorization on integer partitions, transfers of one unit between parts are classical, and closely related descriptions appear already in the literature on cover relations in the dominance lattice; see \cite{Brylawski1973Lattice} for general background, \cite[Lemma~1]{GreeneKleitman1986} for a precise cover criterion in terms of cell transfers, and \cite[Proposition~1]{Ganter2020Notes} for a compact modern formulation. The point of the present paper is therefore not to introduce a new basic move, but to describe the local graph structure generated by all admissible such moves at a fixed vertex.

A second nearby context is that of combinatorial Gray codes on partitions. There, closely related transfer rules are used as minimal-change operations in Hamiltonian path and cycle problems. In particular, this viewpoint goes back to Savage \cite{Savage1989Gray}; for restricted families of partitions see Rasmussen, Savage, and West \cite{RasmussenSavageWest1995}, and for a recent survey see M\"utze \cite{Mutze2023GraySurvey}. Our purpose here is different: we do not study Hamiltonicity or enumeration, but the local morphology of the transfer graph itself.

Our aim is not to revisit the global topology of \(K_n\), but to extract from the transfer rule a compact and usable description of the neighborhood of a partition. The main result shows that several local invariants of a partition \(\lambda\)---namely its admissible transfers, its induced neighborhood graph, its degree, and its local clique data---are determined by a simple ordered binary datum attached to the block structure of \(\lambda\). Concretely, if
\[
\lambda=(s_1^{m_1},\dots,s_t^{m_t}),
\qquad s_1>\dots>s_t>0,
\]
then the relevant local information is encoded by the support size \(t\), the indicator of whether each support block is a singleton block, and the indicator of whether each consecutive support gap is equal to \(1\). From these data one obtains a canonical bipartite graph \(B(\lambda)\), and the neighborhood graph of \(\lambda\) in \(G_n\) is shown to be isomorphic to the line graph \(L(B(\lambda))\).

This yields a uniform description of several local invariants. In particular, the degree of \(\lambda\) admits a closed formula, every clique through \(\lambda\) is shown to be of one of two local types, and the maximal dimension of a simplex of \(K_n\) containing \(\lambda\) can be read off directly from the same transfer data. Thus the local study of \(G_n\) reduces to a finite combinatorial model attached to the ordered support pattern of the partition.

The paper should be viewed as part of a broader program concerning the geometry and morphology of partition graphs. In the previous paper \cite{Lyudogovskiy2026Clique}, the emphasis was on the global homotopy type of the clique complex \(K_n\). Here the emphasis is on the first local layer of the graph \(G_n\). Later stages of the program are intended to address larger-scale structures in \(G_n\), such as central regions, core-like subgraphs, and growth phenomena as \(n\) varies.

The relation to \cite{Lyudogovskiy2026Clique} should therefore be understood as follows. Some of the local mechanisms appearing here are compatible with, and partly motivated by, the earlier classification of cliques in \(G_n\) via star- and top-simplices. However, the present paper is organized independently around the local transfer graph at a single vertex, and the main theorem is stated entirely in graph-theoretic terms. No global homotopy statement is used in the arguments below.

The paper is organized as follows. In Section~2 we fix notation for partitions, corners, admissible transfers, and conjugate coordinates. In Section~3 we determine the admissible transfers from a fixed partition and derive the degree formula. In Section~4 we prove the compatibility criterion for two admissible transfers and identify the induced neighborhood graph with a line graph. In Section~5 we classify local cliques and define the local clique invariants. In Section~6 we introduce the ordered local transfer type. Section~7 contains the main local structure theorem and its corollaries. Section~8 presents several examples.

\section{Definitions and setup}

Let \(n\ge 1\). The partition graph \(G_n\) is the graph whose vertices are the partitions of \(n\), in which two distinct partitions \(\mu,\nu\vdash n\) are adjacent if \(\nu\) is obtained from \(\mu\) by removing one removable corner of the Ferrers diagram of \(\mu\), adding one cell at an addable corner of \(\mu\), and reordering, provided that the resulting partition is different from \(\mu\).

Equivalently, an edge of \(G_n\) is given by one elementary transfer of one cell from one row to another row, where the target row may also be a new row of length \(0\), in which case the transfer creates a new part \(1\).

We write a partition \(\lambda\vdash n\) in block form
\[
\lambda=(s_1^{m_1},s_2^{m_2},\dots,s_t^{m_t}),
\qquad
s_1>s_2>\cdots>s_t>0,
\qquad
m_i\ge 1.
\]
Thus \(s_1,\dots,s_t\) are the distinct part sizes of \(\lambda\), and \(m_i\) is the multiplicity of the part size \(s_i\).

We define
\[
\operatorname{supp}(\lambda):=\{s_1,\dots,s_t\},
\qquad
t=t(\lambda):=|\operatorname{supp}(\lambda)|.
\]
We also set
\[
s_{t+1}:=0,
\qquad
g_i:=s_i-s_{i+1}
\quad (1\le i\le t).
\]
The integers \(g_i\) will be called the \emph{support gaps} of \(\lambda\).

For each \(i\in\{1,\dots,t\}\), let \(c_i\) denote the removable corner corresponding to the block \(s_i^{m_i}\), that is, to the last row of length \(s_i\). Thus the removable corners of \(\lambda\) are naturally indexed by the distinct part sizes.

Likewise, let \(a_j\) (\(1\le j\le t+1\)) denote the addable corners of \(\lambda\). We index them from top to bottom as follows:
\begin{itemize}[leftmargin=2em]
\item for \(1\le j\le t\), the corner \(a_j\) corresponds to increasing one part of size \(s_j\) to \(s_j+1\);
\item the corner \(a_{t+1}\) corresponds to creating a new part \(1\).
\end{itemize}
Thus \(c_i\) is the unique removable corner associated with the block of size \(s_i\), while the addable corners record all possible ways of increasing one support block or creating a new part of size \(1\).

Hence \(\lambda\) has exactly \(t\) removable corners and \(t+1\) addable corners.

For later use, note that the column numbers of these corners are
\[
\operatorname{col}(c_i)=s_i
\qquad (1\le i\le t),
\]
and
\[
\operatorname{col}(a_j)=s_j+1
\qquad (1\le j\le t+1),
\]
where \(s_{t+1}=0\). In particular, distinct removable corners have distinct column numbers, and distinct addable corners also have distinct column numbers.

For \(1\le i\le t\) and \(1\le j\le t+1\), we write
\[
\lambda(c_i\to a_j)
\]
for the partition obtained from \(\lambda\) by removing the corner \(c_i\), adding one cell at the corner \(a_j\), and reordering, provided that the resulting partition is different from \(\lambda\). Such a transfer will be called \emph{admissible}.

If the transfer is admissible, we also write
\[
\lambda_{ij}:=\lambda(c_i\to a_j).
\]

\begin{definition}
The \emph{local admissibility graph} of \(\lambda\), denoted by \(B(\lambda)\), is the bipartite graph with left vertex set
\[
R(\lambda):=\{1,\dots,t\},
\]
right vertex set
\[
A(\lambda):=\{1,\dots,t+1\},
\]
and edge set defined by
\[
i\sim j
\quad\Longleftrightarrow\quad
\lambda(c_i\to a_j)\ \text{is admissible}.
\]
Thus the edges of \(B(\lambda)\) encode the admissible transfers from \(\lambda\).
\end{definition}

We shall use the following conjugate-coordinate description of an elementary transfer.

\begin{lemma}
Let \(\lambda\vdash n\), and let \(\lambda(c\to a)\) be an admissible transfer. Then
\[
(\lambda(c\to a))'
=
\lambda'-e_{\operatorname{col}(c)}+e_{\operatorname{col}(a)},
\]
where \(e_k\) denotes the \(k\)-th standard basis vector.
\end{lemma}

\begin{proof}
Passing from \(\lambda\) to \(\lambda(c\to a)\) removes one cell from the column \(\operatorname{col}(c)\) and adds one cell to the column \(\operatorname{col}(a)\). In the conjugate partition, this decreases the part \(\lambda'_{\operatorname{col}(c)}\) by \(1\) and increases the part \(\lambda'_{\operatorname{col}(a)}\) by \(1\). Hence
\[
(\lambda(c\to a))'
=
\lambda'-e_{\operatorname{col}(c)}+e_{\operatorname{col}(a)}.
\qedhere
\]
\end{proof}

\begin{proposition}
Two distinct partitions \(\mu,\nu\vdash n\) are adjacent in \(G_n\) if and only if
\[
\nu'=\mu'-e_u+e_v
\]
for some \(u\neq v\).
\end{proposition}

\begin{proof}
Assume first that \(\mu\) and \(\nu\) are adjacent in \(G_n\). By definition, \(\nu\) is obtained from \(\mu\) by removing one removable corner of the Ferrers diagram of \(\mu\), adding one cell at an addable corner of \(\mu\), and reordering. In the Ferrers diagram, this removes one cell from column \(u\) and adds one cell to column \(v\), where \(u\neq v\). Therefore, in conjugate coordinates,
\[
\nu'=\mu'-e_u+e_v
\]
for some \(u\neq v\).

Conversely, assume that
\[
\nu'=\mu'-e_u+e_v
\qquad (u\neq v).
\]
Then \(\nu'\) is obtained from \(\mu'\) by decreasing the \(u\)-th part by \(1\) and increasing the \(v\)-th part by \(1\).

Since \(\nu'\) is again a partition, the decrease at coordinate \(u\) is possible only if
\[
\mu'_u>\mu'_{u+1}.
\]
Equivalently, column \(u\) of the Ferrers diagram of \(\mu\) is strictly longer than column \(u+1\), so its terminal cell has no cell to its right. Hence the removed cell is a removable corner of \(\mu\).

Similarly, if \(v>1\), the fact that \(\nu'\) is a partition implies
\[
\mu'_{v-1}>\mu'_v.
\]
Equivalently, column \(v-1\) of the Ferrers diagram of \(\mu\) is strictly longer than column \(v\), so adding one cell at the bottom of column \(v\) produces an addable corner. If \(v=1\), then one adds a cell at the bottom of the first column, which in ordinary partition coordinates corresponds to creating a new part \(1\).

Thus passing from \(\mu\) to \(\nu\) consists of removing one removable corner, adding one cell at an addable corner, and reordering. Therefore \(\mu\) and \(\nu\) are adjacent in \(G_n\).
\end{proof}

\section{Admissible transfers and the degree of a vertex}

We first determine exactly when a transfer is admissible.

\begin{lemma}
Let
\[
\lambda=(s_1^{m_1},\dots,s_t^{m_t})\vdash n.
\]
For \(1\le i\le t\) and \(1\le j\le t+1\), the transfer \(\lambda(c_i\to a_j)\) fails to be admissible if and only if one of the following two cases occurs:
\begin{enumerate}[label=\textup{(\arabic*)},leftmargin=2em]
\item \(j=i\) and \(m_i=1\);
\item \(j=i+1\) and \(g_i=1\).
\end{enumerate}
\end{lemma}

\begin{proof}
Assume first that \(j=i\) and \(m_i=1\). Then \(\lambda\) has only one part of size \(s_i\), so there is no second distinct part of the same size to which the removed cell could be transferred. Hence the move is not admissible.

Now assume that \(j=i+1\) and \(g_i=1\), i.e.
\[
s_i=s_{i+1}+1.
\]
Then the move replaces one part of size \(s_i\) by \(s_i-1=s_{i+1}\), and one part of size \(s_{i+1}\) by \(s_{i+1}+1=s_i\). Therefore the multiset of part sizes is unchanged, so the resulting partition is again \(\lambda\). Hence the move is not admissible.

Conversely, assume that neither \textup{(1)} nor \textup{(2)} occurs.

If \(j=i\), then necessarily \(m_i\ge 2\). The move acts on two distinct parts of size \(s_i\), replacing two occurrences of \(s_i\) by one occurrence of \(s_i-1\) and one occurrence of \(s_i+1\). Since
\[
s_i-1\neq s_i\neq s_i+1,
\]
the multiset of part sizes changes, so the resulting partition is different from \(\lambda\). Hence the move is admissible.

Assume next that \(j\le t\) and \(j\neq i\). Then the move replaces one part of size \(s_i\) by \(s_i-1\) and one part of size \(s_j\) by \(s_j+1\). If the multiset of part sizes were unchanged, then the two altered entries would have to contribute the same two-element multiset as before, that is, \((s_i,s_j)\) rather than \((s_i-1,s_j+1)\). Since \(s_i-1\neq s_i\) and \(s_j+1\neq s_j\), this is possible only if
\[
s_i-1=s_j,
\qquad
s_j+1=s_i.
\]
Hence \(s_i=s_j+1\). Because
\[
s_1>s_2>\dots>s_t,
\]
this can happen only when \(j=i+1\), and then \(g_i=1\), contrary to assumption.

Finally, assume that \(j=t+1\). Then the move replaces one part of size \(s_i\) by \(s_i-1\) and creates a new part \(1\). If the resulting multiset were unchanged, then necessarily \(s_i=1\), so \(i=t\). But then
\[
g_t=s_t-s_{t+1}=1-0=1,
\]
again contrary to assumption.

Thus no other obstruction is possible, and the transfer is admissible.
\end{proof}

\begin{lemma}
Let \(\lambda\vdash n\). Suppose that
\[
\lambda(c\to a)=\lambda(d\to b),
\]
where \(c,d\) are removable corners of \(\lambda\), \(a,b\) are addable corners of \(\lambda\), and both transfers are admissible. Then
\[
c=d,
\qquad
a=b.
\]
\end{lemma}

\begin{proof}
By Lemma~2.2,
\[
(\lambda(c\to a))'
=
\lambda'-e_{\operatorname{col}(c)}+e_{\operatorname{col}(a)},
\]
and
\[
(\lambda(d\to b))'
=
\lambda'-e_{\operatorname{col}(d)}+e_{\operatorname{col}(b)}.
\]
Since the two resulting partitions are equal, their conjugates are equal, and therefore
\[
-e_{\operatorname{col}(c)}+e_{\operatorname{col}(a)}
=
-e_{\operatorname{col}(d)}+e_{\operatorname{col}(b)}.
\]

Because both transfers are admissible, Lemma~3.1 excludes the case \(\operatorname{col}(c)=\operatorname{col}(a)\), and likewise excludes \(\operatorname{col}(d)=\operatorname{col}(b)\). Hence each side is written with exactly one coefficient \(-1\) and exactly one coefficient \(+1\), so the representation is unique. It follows that
\[
\operatorname{col}(c)=\operatorname{col}(d),
\qquad
\operatorname{col}(a)=\operatorname{col}(b).
\]
Since removable corners are uniquely determined by their column numbers, and addable corners are also uniquely determined by their column numbers, we obtain
\[
c=d,
\qquad
a=b.
\]
\end{proof}

\begin{corollary}
The admissible transfers from \(\lambda\) are in bijection with the neighbors of \(\lambda\) in \(G_n\).
\end{corollary}

\begin{proof}
By the definition of \(G_n\), every neighbor of \(\lambda\) is obtained by removing one removable corner of \(\lambda\), adding one cell at one addable corner of \(\lambda\), and reordering. These are precisely the admissible transfers from \(\lambda\). Uniqueness is Lemma~3.2.
\end{proof}

\begin{corollary}
The graph \(B(\lambda)\) is obtained from the complete bipartite graph \(K_{t,t+1}\) by deleting precisely the following edges:
\begin{itemize}[leftmargin=2em]
\item the diagonal edge \(i\!-\!i\) whenever \(m_i=1\);
\item the successor edge \(i\!-\!(i+1)\) whenever \(g_i=1\).
\end{itemize}
\end{corollary}

\begin{proof}
This is an immediate restatement of Lemma~3.1.
\end{proof}

We now encode the two local obstruction patterns by binary indicators.

\begin{definition}
For \(1\le i\le t\), define
\[
\alpha_i:=
\begin{cases}
1,& m_i=1,\\
0,& m_i\ge 2,
\end{cases}
\qquad
\beta_i:=
\begin{cases}
1,& g_i=1,\\
0,& g_i\ge 2.
\end{cases}
\]
We also set
\[
S(\lambda):=\sum_{i=1}^t \alpha_i,
\qquad
U(\lambda):=\sum_{i=1}^t \beta_i.
\]
\end{definition}

\begin{proposition}
The degree of \(\lambda\) in \(G_n\) is
\[
\deg_{G_n}(\lambda)=t(t+1)-S(\lambda)-U(\lambda).
\]
\end{proposition}

\begin{proof}
By Corollary~3.3, the degree of \(\lambda\) is the number of admissible transfers from \(\lambda\), which is the number of edges of \(B(\lambda)\). The complete bipartite graph \(K_{t,t+1}\) has \(t(t+1)\) edges. By Corollary~3.4, exactly \(S(\lambda)\) diagonal edges and \(U(\lambda)\) successor edges are deleted. Therefore
\[
\deg_{G_n}(\lambda)=t(t+1)-S(\lambda)-U(\lambda).
\]
\end{proof}

\begin{proposition}
For \(1\le i\le t\), let
\[
\sigma_i:=\deg_{B(\lambda)}(i),
\]
and for \(1\le j\le t+1\), let
\[
\tau_j:=\deg_{B(\lambda)}(j).
\]
Then
\[
\sigma_i=t+1-\alpha_i-\beta_i
\qquad (1\le i\le t),
\]
and, with the conventions \(\beta_0:=0\) and \(\alpha_{t+1}:=0\),
\[
\tau_j=t-\alpha_j-\beta_{j-1}
\qquad (1\le j\le t+1).
\]
\end{proposition}

\begin{proof}
Fix \(i\). In \(K_{t,t+1}\), the left vertex \(i\) is incident to all \(t+1\) right vertices. By Corollary~3.4, the only possibly missing incident edges are \(i\!-\!i\), deleted exactly when \(\alpha_i=1\), and \(i\!-\!(i+1)\), deleted exactly when \(\beta_i=1\). Hence
\[
\sigma_i=t+1-\alpha_i-\beta_i.
\]

Now fix \(j\). In \(K_{t,t+1}\), the right vertex \(j\) is incident to all \(t\) left vertices. The only possibly missing incident edges are \(j\!-\!j\), deleted when \(j\le t\) and \(\alpha_j=1\), and \((j-1)\!-\!j\), deleted when \(\beta_{j-1}=1\). Hence
\[
\tau_j=t-\alpha_j-\beta_{j-1}.
\]
\end{proof}

\section{Compatibility of admissible transfers}

We now describe the graph induced by \(G_n\) on the neighborhood of a fixed vertex \(\lambda\).

Let
\[
\mu_1=\lambda(c_1\to a_1),
\qquad
\mu_2=\lambda(c_2\to a_2)
\]
be two distinct neighbors of \(\lambda\).

\begin{lemma}
The vertices \(\mu_1\) and \(\mu_2\) are adjacent in \(G_n\) if and only if
\[
c_1=c_2
\qquad\text{or}\qquad
a_1=a_2.
\]
\end{lemma}

\begin{proof}
We distinguish three cases.

First assume that \(c_1=c_2\). By Lemma~2.2,
\[
\mu_1'
=
\lambda'-e_{\operatorname{col}(c_1)}+e_{\operatorname{col}(a_1)},
\qquad
\mu_2'
=
\lambda'-e_{\operatorname{col}(c_1)}+e_{\operatorname{col}(a_2)}.
\]
Hence
\[
\mu_2'-\mu_1'
=
-e_{\operatorname{col}(a_1)}+e_{\operatorname{col}(a_2)}.
\]
Since \(a_1\neq a_2\), the two column numbers are distinct, so Proposition~2.3 implies that \(\mu_1\) and \(\mu_2\) are adjacent in \(G_n\).

Next assume that \(a_1=a_2\). Then
\[
\mu_1'
=
\lambda'-e_{\operatorname{col}(c_1)}+e_{\operatorname{col}(a_1)},
\qquad
\mu_2'
=
\lambda'-e_{\operatorname{col}(c_2)}+e_{\operatorname{col}(a_1)}.
\]
Therefore
\[
\mu_2'-\mu_1'
=
-e_{\operatorname{col}(c_2)}+e_{\operatorname{col}(c_1)}.
\]
Since \(c_1\neq c_2\), the two column numbers are distinct, so Proposition~2.3 again implies adjacency.

Finally assume that
\[
c_1\neq c_2,
\qquad
a_1\neq a_2.
\]
Set
\[
k_r:=\operatorname{col}(c_r),
\qquad
\ell_r:=\operatorname{col}(a_r)
\qquad (r=1,2).
\]
Then
\[
k_1\neq k_2,
\qquad
\ell_1\neq \ell_2.
\]
Also, since both transfers are admissible, Lemma~3.1 implies
\[
k_1\neq \ell_1,
\qquad
k_2\neq \ell_2.
\]
By Lemma~2.2,
\[
\mu_2'-\mu_1'
=
e_{k_1}-e_{\ell_1}-e_{k_2}+e_{\ell_2}.
\]
If \(\mu_1\) and \(\mu_2\) were adjacent, Proposition~2.3 would imply that this vector is of the form
\[
-e_u+e_v
\qquad (u\neq v),
\]
that is, it has exactly one coefficient \(-1\) and exactly one coefficient \(+1\).

This is impossible. Indeed, the equalities \(k_1=\ell_1\) and \(k_2=\ell_2\) are excluded by admissibility, while \(k_1=k_2\) and \(\ell_1=\ell_2\) are excluded by the assumptions \(c_1\neq c_2\) and \(a_1\neq a_2\). Hence the only possible identifications among the four indices are the cross-equalities \(k_1=\ell_2\) and/or \(k_2=\ell_1\). These can only merge coefficients of the same sign. Consequently, after simplification, the vector \(e_{k_1}-e_{\ell_1}-e_{k_2}+e_{\ell_2}\) is either supported on four coordinates, or has a coefficient \(2\) or \(-2\), or is of the form \(2e_p-2e_q\). In none of these cases can it be of the form \(-e_u+e_v\).

Hence \(\mu_1\) and \(\mu_2\) are not adjacent.
\end{proof}

\begin{theorem}
Let
\[
\mathcal N(\lambda):=G_n[N_{G_n}(\lambda)]
\]
be the graph induced on the neighborhood of \(\lambda\). Then
\[
\mathcal N(\lambda)\cong L(B(\lambda)),
\]
the line graph of the local admissibility graph \(B(\lambda)\).
\end{theorem}

\begin{proof}
By Corollary~3.3, admissible transfers from \(\lambda\) are in bijection with the vertices of \(\mathcal N(\lambda)\). By Definition~2.1, the same admissible transfers are in bijection with the edges of \(B(\lambda)\).

Under this identification, Lemma~4.1 shows that two vertices of \(\mathcal N(\lambda)\) are adjacent if and only if the corresponding admissible transfers share the same removable corner or the same addable corner, that is, if and only if the corresponding edges of \(B(\lambda)\) share an endpoint. This is exactly the adjacency relation in the line graph \(L(B(\lambda))\).
\end{proof}

\section{Local cliques and local dimension}

We now translate the line-graph description into clique language.

\begin{lemma}
Let \(H\) be a bipartite graph. Then every set of pairwise adjacent edges of \(H\) has a common endpoint.
\end{lemma}

\begin{proof}
Let \(E_0\) be a set of pairwise adjacent edges of \(H\). Choose one edge
\[
e=(x,y)\in E_0,
\]
where \(x\) lies in the left part of \(H\) and \(y\) lies in the right part.

Since every edge in \(E_0\) is adjacent to \(e\), every edge in \(E_0\) is incident either to \(x\) or to \(y\).

Suppose that not all edges in \(E_0\) share a common endpoint. Then there exists an edge
\[
e_1=(x_1,y)\in E_0
\quad\text{with}\quad x_1\neq x,
\]
and also an edge
\[
e_2=(x,y_2)\in E_0
\quad\text{with}\quad y_2\neq y.
\]
But then \(e_1\) and \(e_2\) are disjoint, since they have different left endpoints and different right endpoints. This contradicts the assumption that all edges in \(E_0\) are pairwise adjacent.

Therefore all edges in \(E_0\) share a common endpoint.
\end{proof}

\begin{corollary}
Every clique of \(G_n\) containing \(\lambda\) is obtained in one of the following two ways:
\begin{enumerate}[label=\textup{(\arabic*)},leftmargin=2em]
\item all non-central vertices are obtained by fixing one removable corner and varying the admissible addable corners;
\item all non-central vertices are obtained by fixing one addable corner and varying the admissible removable corners.
\end{enumerate}
\end{corollary}

\begin{proof}
By Theorem~4.2, a clique containing \(\lambda\) corresponds to a set of pairwise adjacent edges in the bipartite graph \(B(\lambda)\). By Lemma~5.1, all these edges have a common endpoint. Hence the clique is obtained either by fixing one removable corner or by fixing one addable corner.
\end{proof}

These two kinds of cliques will be referred to as \emph{star-type} and \emph{top-type} cliques, respectively.

\begin{definition}
The \emph{local clique number} of \(\lambda\) is
\[
\omega_{\mathrm{loc}}(\lambda)
:=
\max\{\,|C|:\lambda\in C,\ C\text{ is a clique in }G_n\,\}.
\]
The \emph{local simplex dimension} of \(\lambda\) is defined by
\[
\dim_{\mathrm{loc}}(\lambda):=\omega_{\mathrm{loc}}(\lambda)-1.
\]
Thus \(\dim_{\mathrm{loc}}(\lambda)\) is the maximal dimension of a simplex of the clique complex \(K_n=\mathrm{Cl}(G_n)\) containing \(\lambda\).
\end{definition}

\begin{proposition}
For every partition \(\lambda\),
\[
\omega_{\mathrm{loc}}(\lambda)
=
1+\max\Bigl\{\max_{1\le i\le t}\sigma_i,\ \max_{1\le j\le t+1}\tau_j\Bigr\},
\]
and therefore
\[
\dim_{\mathrm{loc}}(\lambda)
=
\max\Bigl\{\max_{1\le i\le t}\sigma_i,\ \max_{1\le j\le t+1}\tau_j\Bigr\}.
\]
\end{proposition}

\begin{proof}
By Corollary~5.2, every clique containing \(\lambda\) is obtained either by fixing one removable corner and taking all admissible addable corners, or by fixing one addable corner and taking all admissible removable corners. The largest clique of the first kind has size \(1+\sigma_i\), and the largest clique of the second kind has size \(1+\tau_j\). Taking the maximum over all \(i\) and \(j\) yields the formula.
\end{proof}

\section{Local transfer type}

The preceding results show that the local combinatorics at \(\lambda\) is controlled by a small amount of data.

\begin{definition}
The \emph{ordered local transfer type} of \(\lambda\) is
\[
\mathfrak t(\lambda)
:=
\bigl(t;\alpha_1,\dots,\alpha_t;\beta_1,\dots,\beta_t\bigr).
\]
Thus \(\mathfrak t(\lambda)\) records:
\begin{itemize}[leftmargin=2em]
\item the support size \(t\);
\item which support blocks are singleton blocks;
\item which consecutive support gaps are unit gaps.
\end{itemize}
\end{definition}

\begin{proposition}
The graph \(B(\lambda)\) is completely determined by \(\mathfrak t(\lambda)\). More precisely, \(B(\lambda)\) is obtained from \(K_{t,t+1}\) by deleting
\begin{itemize}[leftmargin=2em]
\item the edge \(i\!-\!i\) whenever \(\alpha_i=1\);
\item the edge \(i\!-\!(i+1)\) whenever \(\beta_i=1\).
\end{itemize}
\end{proposition}

\begin{proof}
This is exactly Corollary~3.4 rewritten in the notation of Definition~6.1.
\end{proof}

\section{Main local structure theorem}

We now summarize the local structure in one statement.

\begin{theorem}
Let
\[
\lambda=(s_1^{m_1},\dots,s_t^{m_t})\vdash n,
\qquad
s_1>\cdots>s_t>0,
\qquad
s_{t+1}=0.
\]
Define
\[
\alpha_i=\mathbf 1_{(m_i=1)},
\qquad
\beta_i=\mathbf 1_{(s_i-s_{i+1}=1)}
\qquad (1\le i\le t).
\]
Then the ordered local transfer type
\[
\mathfrak t(\lambda)
=
\bigl(t;\alpha_1,\dots,\alpha_t;\beta_1,\dots,\beta_t\bigr)
\]
completely determines the following objects:
\begin{enumerate}[label=\textup{(\arabic*)},leftmargin=2em]
\item the local admissibility graph \(B(\lambda)\);
\item the induced neighborhood graph \(\mathcal N(\lambda)=G_n[N_{G_n}(\lambda)]\);
\item the degree of \(\lambda\) in \(G_n\);
\item the local clique number \(\omega_{\mathrm{loc}}(\lambda)\);
\item the local simplex dimension \(\dim_{\mathrm{loc}}(\lambda)\).
\end{enumerate}

More precisely:
\begin{itemize}[leftmargin=2em]
\item \(B(\lambda)\) is obtained from \(K_{t,t+1}\) by deleting exactly the diagonal edges \(i\!-\!i\) for which \(\alpha_i=1\) and the successor edges \(i\!-\!(i+1)\) for which \(\beta_i=1\);
\item \(\mathcal N(\lambda)\cong L(B(\lambda))\);
\item every clique in \(G_n\) containing \(\lambda\) is either of star type or of top type;
\item one has
\[
\deg_{G_n}(\lambda)
=
t(t+1)-\sum_{i=1}^t\alpha_i-\sum_{i=1}^t\beta_i,
\]
and
\[
\dim_{\mathrm{loc}}(\lambda)
=
\max\Bigl\{
\max_{1\le i\le t}(t+1-\alpha_i-\beta_i),
\ 
\max_{1\le j\le t+1}(t-\alpha_j-\beta_{j-1})
\Bigr\},
\]
with the conventions \(\beta_0:=0\) and \(\alpha_{t+1}:=0\).
\end{itemize}
\end{theorem}

\begin{proof}
The description of \(B(\lambda)\) is Proposition~6.2. The isomorphism
\[
\mathcal N(\lambda)\cong L(B(\lambda))
\]
is Theorem~4.2. The classification of cliques through \(\lambda\) is Corollary~5.2. The degree formula is Proposition~3.6. The formula for \(\dim_{\mathrm{loc}}(\lambda)\) follows from Proposition~3.7 and Proposition~5.4.
\end{proof}

\begin{corollary}
If two partitions \(\lambda\) and \(\mu\) have the same ordered local transfer type,
\[
\mathfrak t(\lambda)=\mathfrak t(\mu),
\]
then their local admissibility graphs are isomorphic, their induced neighborhood graphs are isomorphic, and they have the same degree, the same local clique number, and the same local simplex dimension.
\end{corollary}

\begin{proof}
By Proposition~6.2, the graphs \(B(\lambda)\) and \(B(\mu)\) are isomorphic. Applying Theorem~4.2 and Theorem~7.1 yields the remaining claims.
\end{proof}

\section{Examples}

The following examples illustrate the local transfer type and the resulting local invariants.

\begin{example}[The one-part partition]
Let \(n\ge 2\), and let
\[
\lambda=(n).
\]
Then \(t=1\), \(m_1=1\), and \(g_1=n\). Hence
\[
\alpha_1=1,
\qquad
\beta_1=0.
\]
Therefore \(B(\lambda)\) is obtained from \(K_{1,2}\) by deleting the diagonal edge \(1\!-\!1\). Thus \(B(\lambda)\) has a single edge, corresponding to the unique admissible transfer
\[
(n)\longmapsto (n-1,1).
\]
Accordingly,
\[
\deg_{G_n}(\lambda)=1,
\qquad
\omega_{\mathrm{loc}}(\lambda)=2,
\qquad
\dim_{\mathrm{loc}}(\lambda)=1.
\]
The induced neighborhood graph \(\mathcal N(\lambda)\) consists of a single vertex.
\end{example}

\begin{example}[The all-ones partition]
Let \(n\ge 2\), and let
\[
\lambda=(1^n).
\]
Then \(t=1\), \(m_1=n\), and
\[
g_1=s_1-s_2=1-0=1.
\]
Hence
\[
\alpha_1=0,
\qquad
\beta_1=1.
\]
Therefore \(B(\lambda)\) is obtained from \(K_{1,2}\) by deleting the successor edge \(1\!-\!2\). Again \(B(\lambda)\) has a single edge, now corresponding to the transfer
\[
(1^n)\longmapsto (2,1^{n-2}).
\]
Thus
\[
\deg_{G_n}(\lambda)=1,
\qquad
\omega_{\mathrm{loc}}(\lambda)=2,
\qquad
\dim_{\mathrm{loc}}(\lambda)=1.
\]
This is the dual analogue of Example~8.1.
\end{example}

\begin{example}[A rectangular partition]
Let
\[
\lambda=(r^m),
\qquad r\ge 2,\ m\ge 2.
\]
Then \(t=1\), \(m_1=m\ge 2\), and \(g_1=r\ge 2\). Hence
\[
\alpha_1=0,
\qquad
\beta_1=0.
\]
Therefore \(B(\lambda)=K_{1,2}\). In particular,
\[
\deg_{G_n}(\lambda)=2,
\qquad
\sigma_1=2,
\qquad
\tau_1=\tau_2=1.
\]
Hence
\[
\omega_{\mathrm{loc}}(\lambda)=3,
\qquad
\dim_{\mathrm{loc}}(\lambda)=2.
\]
The two neighbors of \(\lambda\) are adjacent, so \(\lambda\) belongs to a triangle in \(G_n\).
\end{example}

\begin{example}[A partition with no deleted edges in higher support]
Let
\[
\lambda=(4,4,2,2).
\]
Then
\[
\lambda=(4^2,2^2),
\qquad
t=2,
\qquad
(\alpha_1,\alpha_2)=(0,0),
\qquad
(\beta_1,\beta_2)=(0,0),
\]
since the support gaps are \(4-2=2\) and \(2-0=2\). Thus
\[
B(\lambda)=K_{2,3}.
\]
Consequently,
\[
\deg_{G_n}(\lambda)=2\cdot 3=6.
\]
Moreover,
\[
\sigma_1=\sigma_2=3,
\qquad
\tau_1=\tau_2=\tau_3=2.
\]
Hence
\[
\omega_{\mathrm{loc}}(\lambda)=1+\max\{3,2\}=4,
\qquad
\dim_{\mathrm{loc}}(\lambda)=3.
\]
In this case the maximal cliques through \(\lambda\) are of star type.
\end{example}

\begin{example}[The staircase partition]
Let
\[
\delta_t:=(t,t-1,\dots,2,1),
\qquad
n=\frac{t(t+1)}2.
\]
Then \(\delta_t\) is self-conjugate, its support size is \(t\), and every support block is a singleton block. Also every support gap is equal to \(1\). Therefore
\[
\alpha_i=1,
\qquad
\beta_i=1
\qquad
(1\le i\le t).
\]
Hence \(B(\delta_t)\) is obtained from \(K_{t,t+1}\) by deleting every diagonal edge \(i\!-\!i\) and every successor edge \(i\!-\!(i+1)\). In particular,
\[
\deg_{G_n}(\delta_t)=t(t+1)-t-t=t(t-1).
\]
Furthermore,
\[
\sigma_i=t-1
\qquad (1\le i\le t),
\]
while
\[
\tau_1=\tau_{t+1}=t-1,
\qquad
\tau_j=t-2
\qquad (2\le j\le t).
\]
Therefore
\[
\omega_{\mathrm{loc}}(\delta_t)=t,
\qquad
\dim_{\mathrm{loc}}(\delta_t)=t-1.
\]
This example shows that local transfer types with many forbidden edges may still have comparatively large local simplex dimension.
\end{example}

\section{Concluding remarks}

The preceding theorem reduces the local study of the partition graph \(G_n\) to a finite combinatorial model. In particular, the local invariants considered here depend only on three kinds of data: the support size, the singleton-block pattern, and the unit-gap pattern. This provides a natural local language for the subsequent study of larger-scale structures in \(G_n\).

\section*{Acknowledgements}

The author acknowledges the use of ChatGPT (OpenAI) for discussion, structural planning, and editorial assistance during the preparation of this manuscript. All mathematical statements, proofs, computations, and final wording were checked and approved by the author, who takes full responsibility for the contents of the paper.


\begin{thebibliography}{99}

\bibitem{Lyudogovskiy2026Clique}
F.~B.~Lyudogovskiy,
\newblock \emph{The homotopy type of the clique complex of the partition graph},
\newblock arXiv:2603.14370 [math.CO], 2026.

\bibitem{Brylawski1973Lattice}
T.~Brylawski,
\newblock The lattice of integer partitions,
\newblock \emph{Discrete Mathematics} \textbf{6} (1973), no.~3, 201--219.
\newblock DOI: \texttt{10.1016/0012-365X(73)90094-0}.

\bibitem{GreeneKleitman1986}
C.~Greene and D.~J.~Kleitman,
\newblock Longest chains in the lattice of integer partitions ordered by majorization,
\newblock \emph{European Journal of Combinatorics} \textbf{7} (1986), no.~1, 1--10.
\newblock DOI: \texttt{10.1016/S0195-6698(86)80013-0}.

\bibitem{Ganter2020Notes}
B.~Ganter,
\newblock Notes on integer partitions,
\newblock in: \emph{Concept Lattices and Their Applications (CLA 2020)},
\newblock CEUR Workshop Proceedings \textbf{2668} (2020), 19--31.

\bibitem{Savage1989Gray}
C.~D.~Savage,
\newblock Gray code sequences of partitions,
\newblock \emph{Journal of Algorithms} \textbf{10} (1989), no.~4, 577--595.
\newblock DOI: \texttt{10.1016/0196-6774(89)90007-2}.

\bibitem{RasmussenSavageWest1995}
D.~Rasmussen, C.~D.~Savage, and D.~B.~West,
\newblock Gray code enumeration of families of integer partitions,
\newblock \emph{Journal of Combinatorial Theory, Series A} \textbf{70} (1995), no.~2, 201--229.
\newblock DOI: \texttt{10.1016/0097-3165(95)90090-X}.

\bibitem{Mutze2023GraySurvey}
T.~M\"utze,
\newblock Combinatorial Gray codes --- an updated survey,
\newblock \emph{Electronic Journal of Combinatorics} \textbf{30} (2023), no.~3, Dynamic Survey~DS26.
\newblock DOI: \texttt{10.37236/11023}.

\end{thebibliography}
\end{document}